\documentclass[a4paper, 11pt, twoside]{article}

\pdfoutput=1

\usepackage[utf8]{inputenc}
\usepackage{enumerate}
\usepackage[ngerman, english]{babel}
\usepackage{array}
\usepackage{amsmath}
\usepackage{amssymb}
\usepackage{amsthm}
\usepackage{graphicx} 
\usepackage{setspace}
\usepackage{color}
\usepackage{scrlayer-scrpage}
\usepackage{titlesec}
\usepackage[pdftex]{hyperref}
\usepackage[a4paper, lmargin=3.5cm, rmargin=3.0cm, top=3.0cm, bottom=2.5cm, head=14.5pt]{geometry}

 \hypersetup{
	pdftitle = {Exponential polynomials with Fatou and non-escaping sets of finite Lebesgue measure},
	pdfauthor = {Mareike Wolff},
	colorlinks = true,  
	linkcolor = black,
	citecolor = black,
	}		
  
\newcommand\dC{\mathbb{C}}

\newcommand\dN{\mathbb{N}}

\newcommand\dZ{\mathbb{Z}}

\newcommand\cK{\mathcal{K}}

\newcommand\cO{\mathcal{O}}

\newcommand\cS{\mathcal{S}}

\DeclareMathOperator{\re}{Re}
\DeclareMathOperator{\im}{Im}
\DeclareMathOperator{\dens}{dens}
\DeclareMathOperator{\meas}{meas}
\DeclareMathOperator{\dist}{dist}
\DeclareMathOperator{\diam}{diam}
\DeclareMathOperator{\pack}{pack}
\DeclareMathOperator{\ann}{ann}

\newtheorem{thm}{Theorem}[section]
\newtheorem{lemma}[thm]{Lemma}
\newtheorem{cor}[thm]{Corollary}
\newtheorem{example}[thm]{Example}
\theoremstyle{definition}

\newtheorem*{remark}{Remark}

\titleformat{\section}
{\normalfont\normalsize\bfseries}{\thesection}{1em}{}

\cohead{\small Fatou and non-escaping sets of finite measure}
\cehead{\small Mareike Wolff}

\begin{document}
\title{Exponential polynomials with Fatou and non-escaping sets of finite Lebesgue measure}
\author{Mareike Wolff}
\date{}
\maketitle

\begin{abstract}
\noindent We give conditions ensuring that the Fatou set and the complement of the fast escaping set of an exponential polynomial $f$ have finite Lebesgue measure. Essentially, these conditions are designed such that $|f(z)|\ge\exp(|z|^\alpha)$ for some $\alpha>0$ and all $z$ outside a set of finite Lebesgue measure.
\end{abstract}

\section{Introduction and results}

Let $f$ be a transcendental entire function, and let  $f^n$ denote the $n$-th iterate of $f$. The \textit{Fatou set} $F(f)$ of $f$ is  the set of all $z\in\dC$ such that the iterates $(f^n)_{n\in\dN_0}$ form a normal family in a neighborhood of $z$, and the \textit{Julia set} $J(f)$ is the complement of $F(f)$. These sets play an important role in complex dynamics. Clearly, $F(f)$ is open, and $J(f)$ is closed. Moreover, $J(f)$ is always non-empty, and either $J(f)=\dC$, or $J(f)$ has empty interior. An introduction to the dynamics of transcendental entire functions can be found in \cite{bergweiler2}. The \textit{escaping set} of $f$ is defined by
$$I(f):=\{z:\,f^n(z)\to\infty\text{ as }n\to\infty\}.$$
Eremenko \cite{eremenko} showed that $I(f)$ is always non-empty, and that $J(f)=\partial I(f)$. For $r>0$, let $M(r,f):=\max_{|z|=r}|f(z)|$ denote the \textit{maximum modulus} of $f$, and let $M^n(r,f)$ be its $n$-th iterate with respect to $r$. The \textit{fast escaping set} $A(f)$ is a subset of the escaping set. It was introduced by Bergweiler and Hinkkanen \cite{bergweiler-hinkkanen}, and is defined by
$$A(f):=\{z:\, \text{there exists }l\in\dN\text{ such that }|f^n(z)|\ge M^{n-l}(R,f)\text{ for }n>l\},$$
where $R$ is chosen such that $M(r,f)>r$ for $r\ge R$.\\

We are interested in the Lebesgue measure of the sets defined above. McMullen \cite{mcmullen} showed that the Julia set of $f(z)=\sin(az+b),\,a\ne0$, has positive Lebesgue measure. In fact, it can be seen from the proof that also $J(f)\cap A(f)$ has positive measure. Sixsmith \cite{sixsmith} proved that if $f(z)=\sum_{j=1}^qa_j\exp\left(\omega_q^jz\right)$, where $q\ge2,\,a_j\in\dC\setminus\{0\}$, and $\omega_q=\exp(2\pi i /q)$, then $J(f)\cap A(f)$ has positive measure. Sixsmith remarked without proof that his result remains true for 
\begin{equation}  \label{eq_sixsmith}
f(z)=\sum_{j=1}^qa_j\exp(b_jz),
\end{equation}
 where $q\ge3$, $a_j,b_j\in\dC\setminus\{0\}$, $\arg(b_j)<\arg(b_{j+1})<\arg(b_j)+\pi$ for $j\in\{1,...,q-1\}$, and $\arg(b_q)>\arg(b_1)+\pi$, with the argument chosen in $[0,2\pi)$. Bergweiler and Chyzhykov \cite{bergweiler-chyzhykov} gave conditions ensuring that the Julia set and the escaping set of a transcendental entire function of completely regular growth have positive measure. These conditions are satisfied for the functions \eqref{eq_sixsmith}. In fact, they are also satisfied if one allows $\arg(b_{j+1})=\arg(b_j)+\pi$ for some $j\in\{1,...,q-1\}$ or $\arg(b_q)=\arg(b_1)+\pi$. Further criteria for Julia sets and (fast) escaping sets to have positive measure are given in \cite{aspenberg-bergweiler, bergweiler}.\\
 
For certain functions it is possible to obtain stronger results in the sense that one can bound the size of the complement of the Julia set or (fast) escaping set. Schubert \cite{schubert} used McMullen's methods to prove that if $f(z)=\sinh(z)$, then the Lebesgue measure of $F(f)$ and $\dC\setminus I(f)$ is finite in any horizontal strip of width $2\pi$. In fact, the proof shows that one may replace $I(f)$ by the fast escaping set $A(f)$ here. Schubert's result was generalized by Zhang and Yang \cite{zhang-yang} to functions of the form $f(z)=P(e^z)/e^z$, where $P$ is a polynomial of degree at least $2$ satisfying $P(0)\ne0$.\\

There seem to be no papers whose main aim is to show that the Lebesgue measure of the Fatou set or the complement of the (fast) escaping set of certain transcendental entire functions is finite. However, there are some results occurring in papers mainly treating a different subject. We mention two of them. Hemke \cite[Theorem 5.1]{hemke} showed that if 
\begin{equation}    \label{eq_hemke}
f(z)=Q_1(z)\exp(P(z))+Q_2(z)\exp(-P(z)),
\end{equation}
where $P, Q_1, Q_2$ are polynomials with $Q_1,Q_2\not\equiv0$ and $\deg(P)\ge3$, then the Lebesgue measure of $\dC\setminus I(f)$ is finite. One example for such a function is $f(z)=\sin(z^3)$. A result of Bock \cite[Example 2]{bock} says that if $f(z)=\sin(\pi z)$, then $F(f)=\emptyset$, and $(f^n(z))$ tends to infinity for almost all $z\in\dC$. 

This is different for $f(z)=\sin(z)$, which is conjugate to the function $\sinh(z)$ considered by Schubert. From $f(0)=0$ and $f'(0)=1$ it follows that $F(f)\ne\emptyset$, and that there exists  a component of $F(f)$ where $(f^n(z))$ tends to zero. Since $F(f)$ is open and $2\pi$-periodic, the Lebesgue measure of $F(f)$ and $\dC\setminus I(f)$ is infinite. Also, the Fatou set and non-escaping set of $f(z)=\sin(z^2)$ have infinite measure. To see this, note that $f$ has a superattracting fixed point at zero. Let $\varepsilon >0$ such that $D(0,\varepsilon)$ is contained in the attractive basin of zero. There exists $\delta\in(0,\pi/2)$ such that $\sin(D(\pi k,\delta))\subset D(0,\varepsilon)$ for all $k\in\dZ$. Let $D_k:=D(\pi k,\delta)$ and $p(z)=z^2$. Then $p^{-1}(D_k)$ is contained in the attractive basin of zero of the function $f$. For a measurable set $A\subset\dC$ let $\meas(A)$ denote the Lebesgue measure of $A$. We get
$$\meas(p^{-1}(D_k))=2\int_{D_k}\left(\frac{1}{2\sqrt{|z|}}\right)^2d(x,y)\ge\frac{1}{2(|k|\pi+\delta)}\meas(D_k)=\frac{\pi\delta^2}{2(|k|\pi+\delta)}.$$
Summing up over all $k$ yields that the attractive basin of zero has infinite measure. See Figure \ref{fig_sin} for an illustration of the non-escaping sets of $\sin(z)$, $\sin(z^2)$, and $\sin(z^3)$.

\begin{figure}[ht]     
\centering 
\includegraphics[width=\textwidth]{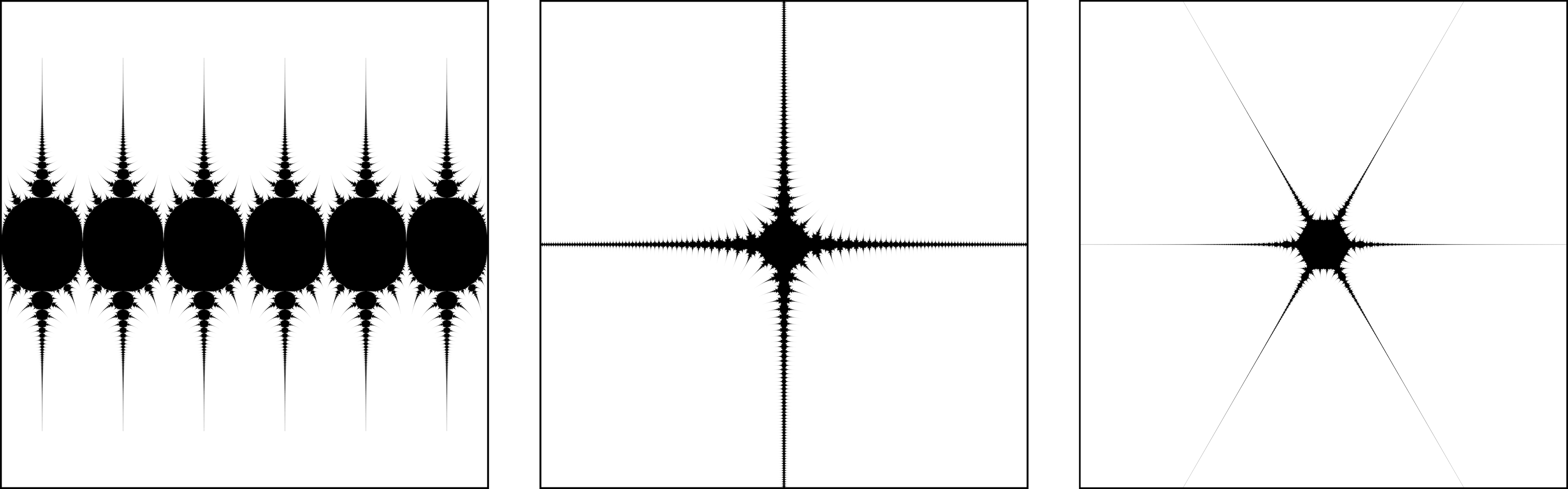}

\caption{The non-escaping sets of $\sin(z)$, $\sin(z^2)$, and $\sin(z^3)$.}
\label{fig_sin}
\end{figure}

In this paper, we consider \textit{exponential polynomials} of the form 
$$f(z)=\sum_{j=1}^N Q_j(z)\exp(b_jz^d+P_j(z)),$$
where $Q_j$ and $P_j$ are polynomials with $\deg(P_j)<d$. We give conditions ensuring that the Lebesgue measure of the complement of $J(f)\cap A(f)$ is finite.

\begin{thm}    \label{thm_meas_short}
Let 
$$f(z):=\sum_{j=1}^N Q_j(z)\exp(b_jz^d+P_j(z)),$$
where $d\in\dN$ with $d\ge3$, $P_j$ and $Q_j$ are polynomials with $Q_j\not\equiv0$ and $\deg(P_j)<d$, and $b_j\in\dC\setminus\{0\}$ are distinct numbers satisfying $\arg(b_j)\le\arg(b_{j+1})<\arg(b_j)+\pi$ for all $j\in\{1,...,N-1\}$ and $\arg(b_N)>\arg(b_1)+\pi$, with the argument chosen in $[0,2\pi)$.
Then the Lebesgue measure of $\dC\setminus(A(f)\cap J(f))$ is finite.
\end{thm}

Note that the conditions on the $b_j$ imply that $N\ge3$. Recall that Sixsmith's result for the functions \eqref{eq_sixsmith} remains true if $\arg(b_{j+1})=\arg(b_j)+\pi$ for some $j\in\{1,...,q-1\}$ or $\arg(b_{q})=\arg(b_1)+\pi$. This is not true in general for Theorem \ref{thm_meas_short}, as the following example shows.

\begin{example}     \label{ex_counterex}
Let 
$$h(z):=\frac{1}{2}\exp(z^3+iz)-\frac{1}{2}\exp(-z^3+iz)=\exp(iz)\sinh(z^3).$$
Then $h$ has a superattracting fixed point at zero, and the attractive basin of zero has infinite Lebesgue measure. In particular, the Lebesgue measure of $\dC\setminus(A(h)\cap J(h))$ is infinite.
\end{example}

We will verify this in Section \ref{sec_counterexample}. However, under certain additional conditions on the polynomials $P_j$, the statement of Theorem \ref{thm_meas_short}  remains true if $\arg(b_{j+1})=\arg(b_j)$ for some $j\in\{1,...,N-1\}$ or $\arg(b_N)=\arg(b_1)+\pi$. This is the following result.

\begin{thm}   \label{thm_meas_detailed}
Let 
$$f(z):=\sum_{j=1}^N Q_j(z)\exp(b_jz^d+P_j(z)),$$
where $d\in\dN$ with $d\ge3$, $P_j$ and $Q_j$ are polynomials with $Q_j\not\equiv0$ and $\deg(P_j)<d$, and $b_j\in\dC\setminus\{0\}$ are distinct numbers satisfying $\arg(b_j)\le\arg(b_{j+1})\le\arg(b_j)+\pi$ for all $j\in\{1,...,N-1\}$ and $\arg(b_N)\ge\arg(b_1)+\pi$, with the argument chosen in $[0,2\pi)$.

If there exists $j\in\{1,...,N-1\}$ such that $\arg(b_{j+1})=\arg(b_{j})+\pi$, or if $\arg(b_N)=\arg(b_1)+\pi$, in addition suppose that 
there are $k,l \in\{1,...,N\}$ with $\arg(b_k)=\arg(b_j)$ and $\arg(b_l)=\arg(b_{j+1})$, or $\arg(b_k)=\arg(b_1)$ and $\arg(b_l)=\arg(b_N)$, respectively, such that
the polynomials $P_k,P_l$ can be written in the form 
\begin{equation}    \label{eq_extra-condition}
P_k(z)=b_kg(z)+g_k(z)\quad\text{and}\quad P_l(z)=b_lg(z)+g_l(z)
\end{equation}
 with polynomials $g,g_k,g_l$ satisfying $\deg(g)\le d-1$ and $\max\{\deg(g_k),\,\deg(g_l)\}\le d-3$.
 
 Then the Lebesgue measure of $\dC\setminus (A(f)\cap J(f))$ is finite.
\end{thm}

Note that the conditions on the $b_j$ imply that  $N\ge2$. Theorem \ref{thm_meas_short} is a special case of Theorem \ref{thm_meas_detailed}. Also, the functions \eqref{eq_hemke} considered by Hemke satisfy the assumptions of Theorem \ref{thm_meas_detailed}. \\

Throughout the rest of the paper, let $f$ be an entire function satisfying the assumptions of Theorem \ref{thm_meas_detailed}. In Section \ref{sec_f}, we will show that $f$ can be approximated by simpler functions in large parts of the complex plane, and use this to prove that $|f(z)|$ is large outside a set of finite measure. Then, in Section \ref{sec_injectivity}, we show that $f$ is injective in certain small disks. We finish the proof of Theorem \ref{thm_meas_detailed} in Section \ref{sec_proof}, using a construction similar to one that occurred in McMullen's paper \cite{mcmullen}, and has since then been used by various authors. Finally, in Section \ref{sec_counterexample}, we verify the properties of Example \ref{ex_counterex}

\section{The behaviour of \texorpdfstring{$f$}{f}}   \label{sec_f}

In this section, we prove several properties of the function $f$. We first introduce some notations. For $j,\,k\in\{1,...,N\}$ with $j\ne k$ let
$$P_{j,k}(z):=(b_jz^d+P_j(z))-(b_kz^d+P_k(z))=(b_j-b_k)z^d+(P_j(z)-P_k(z)).$$
Let 
$$\nu:=d-\frac{5}{2},$$
and define the sets
$$U_1:=\left\{w\in\dC:\, |\re(w)|<|w|^{\nu/d}\right\}$$
and
$$U_2:=\left\{w\in\dC:\,|\re(w)|<2|w|^{\nu/d}\right\}.$$
Moreover, define ``exceptional sets''
$$E_l:=\bigcup_{\substack{j,k=1\\j\ne k}}^NP_{j,k}^{-1}(U_l),$$
for $l\in\{1,2\}$ (see Figure \ref{fig_E1/2}). 

\begin{figure}[ht]     
\centering 
\includegraphics[width=0.45\textwidth]{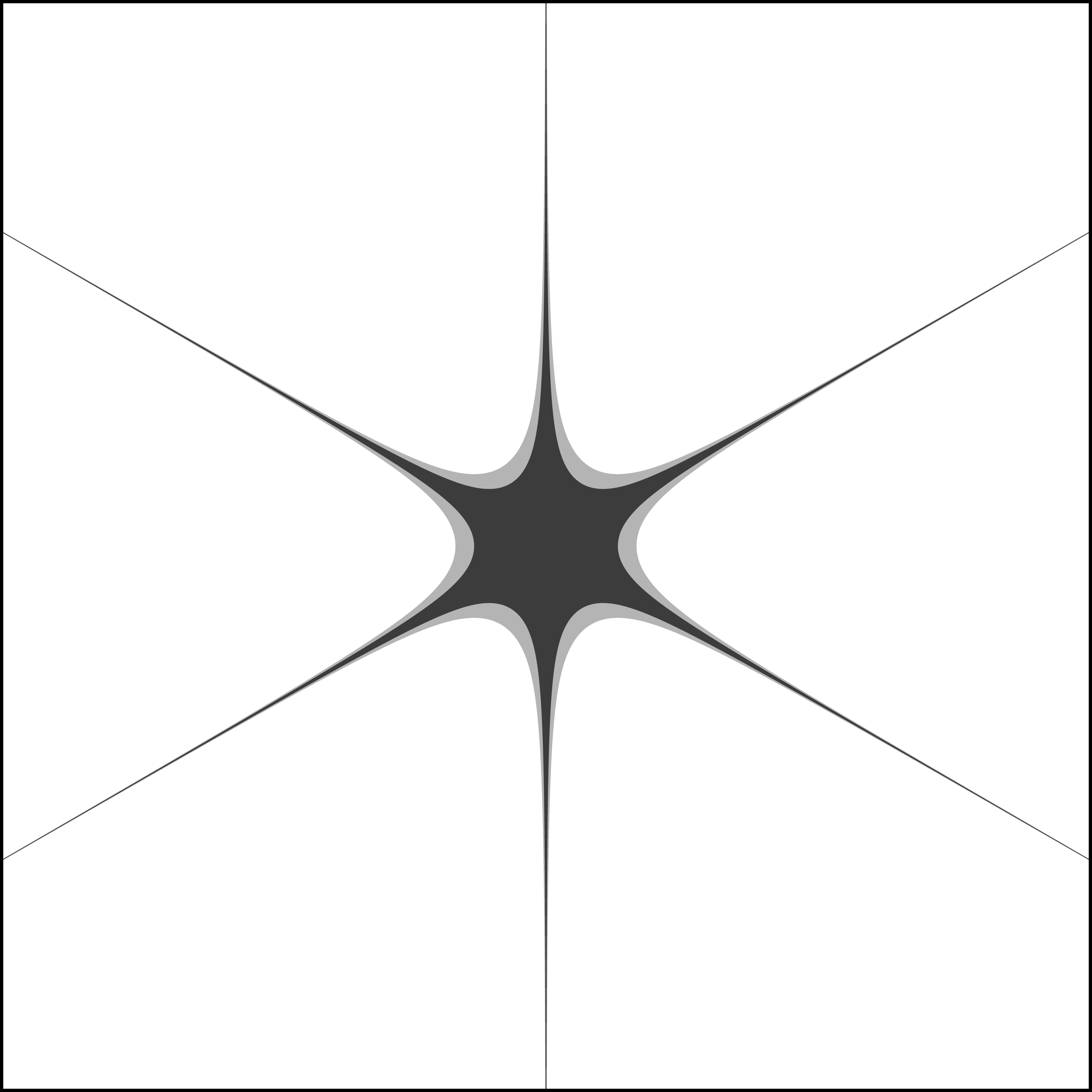}

\caption{The sets $E_1$ (dark grey) and $E_2$ (light and dark grey) for $f(z)=Q_1(z)\exp(z^3)+Q_2(z)\exp(-z^3)$.}
\label{fig_E1/2}
\end{figure}

\begin{lemma}   \label{lemma_measE2}
The Lebesgue measure of $E_1$ and $E_2$ is finite.
\end{lemma}

We will later prove that the function $f$ behaves ``nicely'' outside $E_1$. For $z_0\in\dC$ and $r>0$, we denote by $D(z_0,r):=\{z:\,|z-z_0|<r\}$ the open disk of radius $r$ around $z_0$. Lemma \ref{lemma_measE2} follows directly from

\begin{lemma}
Let $P$ be a polynomial of degree $d$, with $d$ as before. Then the Lebesgue measure of $P^{-1}(U_2)$ is finite.
\end{lemma}

\begin{proof}
Write
$$P(z)=\sum_{j=0}^da_jz^j.$$
Fix $R>2^{2d/5}$ such that all critical values of $P$ are contained in $D(0,R)$, and let $V$ be a component of $P^{-1}\left(\dC\setminus\left(\overline{D(0,R)}\cup(-\infty,0]\right)\right)$. Then the restriction $P|_V:V\to\dC\setminus\left(\overline{D(0,R)}\cup(-\infty,0]\right)$ is biholomorphic. Let $\varphi$ denote the corresponding inverse function. Then $E_2\cap V=\varphi\left(U_2\setminus\overline{D(0,R)}\right)$. Let 
$$W:=U_2\setminus \overline{D(0,R)}=\left\{re^{i\theta}: \,r>R, \,\min\left\{\left|\theta-\frac{\pi}{2}\right|,\,\left|\theta-\frac{3\pi}{2}\right|\right\}<\arcsin(2r^{-5/(2d)})\right\}.$$
We have, with $\eta(r)=\arcsin(2r^{-5/(2d)})$,
\begin{align*}
&\meas(E_2\cap V)\,=\,\int_{W}|\varphi'(z)|^2d(x,y)\,=\,\int_{W}\dfrac{1}{|P'(\varphi(z))|^2}d(x,y)\\
&\qquad=\int_{R}^\infty r\left(\int_{\frac{\pi}{2}-\eta(r)}^{\frac{\pi}{2}+\eta(r)}\dfrac{1}{|P'(\varphi(re^{i\theta}))|^2}\,d\theta+\int_{\frac{3\pi}{2}-\eta(r)}^{\frac{3\pi}{2}+\eta(r)}\dfrac{1}{|P'(\varphi(re^{i\theta}))|^2}\,d\theta\right)\,dr.
\end{align*}
If $r$ is sufficiently large, then 
$$r=|P(\varphi(re^{i\theta}))|\le2^{d/(d-1)}|a_d||\varphi(re^{i\theta})|^d.$$
Thus,
$$|\varphi(re^{i\theta})|\ge2^{-1/(d-1)}|a_d|^{-1/d}\cdot r^{1/d}$$
and
\begin{equation}  \label{eq_P'}
|P'(\varphi(re^{i\theta}))|\ge\frac{1}{2}d|a_d||\varphi(re^{i\theta})|^{d-1}\ge\frac{1}{4}d|a_d|^{1/d}\cdot r^{(d-1)/d}.
\end{equation}
Using \eqref{eq_P'} and $$\eta(r)=\arcsin(2r^{-5/(2d)})\le\pi r^{-5/(2d)},$$
we get
\begin{align*}
 \meas(E_2\cap V)&\le\int_{R}^\infty r\cdot4\eta(r)\cdot \dfrac{16}{d^2|a_d|^{2/d}r^{2(d-1)/d}}\,dr\\
&\le\dfrac{64\pi}{d^2}|a_d|^{-2/d}\int_{R}^\infty r\cdot r^{-5/(2d)}\cdot r^{-2(d-1)/d}\,dr\\
&=\dfrac{64\pi}{d^2}|a_d|^{-2/d}\int_{R}^\infty r^{-(1+1/(2d))}\,dr\,<\infty.
\end{align*}
Since there are only finitely many such components $V$, and also $P^{-1}(\overline{D(0,R)})$ has finite Lebesgue measure, the claim follows.
\end{proof}

The next Lemma yields that if $R_0>0$ is large, then in each component of $\dC\setminus (E_1\cup D(0,R_0))$, $f$ behaves like one of its summands $Q_m(z)\exp(b_mz^d+P_m(z))$.
 
 \begin{lemma}                \label{lemma_approx}
 Let $\varepsilon>0$. If $R_0>0$ is sufficiently large, then for each connected component $V$ of $\dC\setminus (E_1\cup D(0,R_0))$ there is an $m\in\{1,...,N\}$, such that, for all $z\in V$ and all $j\in\{1,...,N\}$ with $j\ne m$, we have
$$\re(b_mz^d+P_m(z))>\re(b_jz^d+P_j(z)),$$
and
$$\left|\frac{f(z)}{Q_m(z)\exp(b_mz^d+P_m(z))}-1\right|<\varepsilon,$$
$$\left|\frac{f'(z)}{db_mz^{d-1}Q_m(z)\exp(b_mz^d+P_m(z))}-1\right|<\varepsilon,$$
$$\left|\frac{f''(z)}{d^2b_m^2z^{2d-2}Q_m(z)\exp(b_mz^d+P_m(z))}-1\right|<\varepsilon.$$
\end{lemma}

\begin{proof}
Let $R_0>0$ such that all zeros of the polynomials $P_{j,k}$  are contained in $D(0,R_0)$. Let $z\in\dC\setminus E_1$ with $|z|\ge R_0$. Then, for all $j,k\in\{1,...,N\}$ with $j\ne k$, 
$$|\re(b_jz^d+P_j(z))-\re(b_kz^d+P_k(z))|=|\re(P_{j,k}(z))|\ge|P_{j,k}(z)|^{\nu/d}.$$
Thus, there exists $m\in\{1,...,N\}$ such that, for all $j\in\{1,...,N\}$ with $j\ne m$,
$$\re(b_mz^d+P_m(z))\ge\re(b_jz^d+P_j(z))+|P_{j,m}(z)|^{\nu/d}>\re(b_jz^d+P_j(z)).$$
By continuity, $m$ depends only on the connected component of $\dC\setminus (E_1\cup D(0,R_0))$ containing $z$, and not on $z$ itself.
We get
\begin{align*}
&\left|\dfrac{f(z)}{Q_m(z)\exp\left(b_mz^d+P_m(z)\right)}-1\right|\\
&\qquad=\left|\dfrac{\sum_{j=1}^NQ_j(z)\exp\left(b_jz^d+P_j(z)\right)-Q_m(z)\exp\left(b_mz^d+P_m(z)\right)}{Q_{m}(z)\exp\left(b_mz^d+P_m(z)\right)}\right|\\
&\qquad=\Bigg|\sum_{\substack{j=1\\j\ne m}}^N\dfrac{Q_j(z)}{Q_m(z)}\exp\left(b_jz^d+P_j(z)-(b_mz^d+P_m(z))\right)\Bigg|\\
&\qquad\le\sum_{\substack{j=1\\j\ne m}}^N\left|\dfrac{Q_j(z)}{Q_m(z)}\right|\exp\left(\re(b_jz^d+P_j(z))-\re(b_mz^d+P_m(z))\right)\\
&\qquad\le\sum_{\substack{j=1\\j\ne m}}^N\left|\dfrac{Q_j(z)}{Q_m(z)}\right|\exp\left(-|P_{j,m}(z)|^{\nu/d}\right)<\varepsilon,
\end{align*}
if $R_0$ and hence $|z|$ is sufficiently large. This is the result for $f$.

Moreover, 
$$f'(z)=\sum_{j=1}^N(db_jz^{d-1}Q_j(z)+P_j'(z)Q_j(z)+Q_j'(z))\exp(b_jz^d+P_j(z).$$
The result for $f'$ now follows from similar estimates as above and the fact that 
$$db_jz^{d-1}Q_j(z)+P_j'(z)Q_j(z)+Q_j'(z)=(1+o(1))db_jz^{d-1}Q_j(z)\quad\text{as }|z|\to\infty.$$
Analogously, the result for $f''$  follows from
$$f''(z)=\sum_{j=1}^N(1+o(1))d^2b_j^2z^{2d-2}Q_j(z)\exp(b_jz^d+P_j(z))\quad\text{as }|z|\to\infty.$$
\end{proof}

\begin{remark}
In order to prove Lemma \ref{lemma_approx}, we did not need any assumptions on the arguments of the $b_j$. In particular, the statement remains true without the additional condition \eqref{eq_extra-condition} in the case that $\arg(b_{j+1})=\arg(b_j)+\pi$ for some $j\in\{1,...,N-1\}$ or $\arg(b_N)=\arg(b_1)+\pi$.
\end{remark}

This is different for the next result.

\begin{lemma}    \label{lemma_flarge}
Let $\alpha\in(0,\nu)$. If $z\in\dC\setminus E_1$ and $|z|$ is sufficiently large, then
$$|f(z)|\ge \exp(|z|^\alpha)\quad\text{and}\quad|f'(z)|\ge \exp(|z|^\alpha).$$
\end{lemma}

\begin{remark}
Without the additional condition \eqref{eq_extra-condition} in the case that $\arg(b_{j+1})=\arg(b_j)+\pi$ for some $j\in\{1,...,N\}$ or $\arg(b_N)=\arg(b_1)+\pi$, the statement of Lemma \ref{lemma_flarge} is not true in general. We will prove in Section \ref{sec_counterexample} that the function $h(z)=\frac{1}{2}\exp(z^3+iz)-\frac{1}{2}\exp(-z^3+iz)$ given in Example \ref{ex_counterex} is bounded in a set of infinite Lebesgue measure.
\end{remark}

\begin{proof} [Proof of Lemma \ref{lemma_flarge}]
We prove the statement for $f$. The proof for $f'$ is analogous. Let $z\in\dC\setminus E_1$, and let $m\in\{1,...,N\}$ with 
$$\re(b_mz^d+P_m(z))=\max_{1\le j\le N}\re(b_jz^d+P_j(z)).$$ By Lemma \ref{lemma_approx},
$$|f(z)|\ge\frac{1}{2}|Q_m(z)|\exp\left(\re(b_mz^d+P_m(z))\right),$$
if $|z|$ is sufficiently large. Thus, it suffices to show that there exists $j\in\{1,...,N\}$ with
$$\re(b_jz^d+P_j(z))\ge 2|z|^\alpha.$$
We first consider the case that $f$ satisfies the assumptions of Theorem \ref{thm_meas_short}, that is, 
$$\arg(b_{j+1})<\arg(b_j)+\pi \text{ for all }j\in\{1,...,N-1\}\text{ and}\, \arg(b_N)>\arg(b_1)+\pi.$$
Then there is a constant $C>0$, such that for all $z\in\dC$ there exists $j\in\{1,...,N\}$ with
$\re(b_jz^d)\ge 2C|z|^d$. Since $\deg(P_j)<d$, we get
$$\re(b_jz^d+P_j(z))\ge 2C|z|^d-|P_j(z)|\ge C|z|^d>2|z|^\alpha,$$
if $|z|$ is sufficiently large. 

Now suppose that $f$ does not satisfy the assumptions of Theorem \ref{thm_meas_short}. Then the assumptions of Theorem \ref{thm_meas_detailed} imply that there are $j,k\in\{1,...,N\}$ satisfying $|\arg(b_j)-\arg(b_k)|=\pi$, and polynomials $g,g_j, g_k$ with $\deg(g)\le d-1$ and $\max\{\deg(g_j),\,\deg(g_k)\}\le d-3$, such that 
$$P_j=b_jg+g_j\quad\text{and}\quad P_k=b_kg+g_k.$$
 With $\beta_j:=\arg(b_j)$ we have
$$b_j=|b_j|e^{i\beta_j}\quad\text{and}\quad b_k=-|b_k|e^{i\beta_j}.$$
Thus,
$$b_j-b_k=(|b_j|+|b_k|)e^{i\beta_j}$$
and
$$P_{j,k}(z)=(|b_j|+|b_k|)e^{i\beta_j}(z^d+g(z))+(g_j(z)-g_k(z)).$$
Moreover, we assume without loss of generality that 
$$\re\left(b_j(z^d+g(z))\right)\ge\re\left(b_k(z^d+g(z))\right).$$ Then $\re(b_j(z^d+g(z)))\ge0$.
Since $z\notin E_1$, we get
$$|P_{j,k}(z)|^{\nu/d}\le|\re(P_{j,k}(z))|\le(|b_j|+|b_k|)\re\left(e^{i\beta_j}(z^d+g(z))\right)+|g_j(z)-g_k(z)|.$$
Hence,
\begin{align*}
\re(b_jz^d+P_j(z))&=\re\left(b_j(z^d+g(z))+g_j(z)\right)\\
&=|b_j|\re\left(e^{i\beta_j}(z^d+g(z))\right)+\re\left(g_j(z)\right)\\
&\ge\frac{|b_j|}{|b_j|+|b_k|}\left(|P_{k,j}(z)|^{\nu/d}-|g_j(z)-g_k(z)|\right)-|g_j(z)|.
\end{align*}
Since $|P_{j,k}(z)|\ge2^{-d/\nu}(|b_j|+|b_k|)|z|^d$ if $|z|$ is large, and since $\max\{\deg(g_j),\,\deg(g_k)\}<\nu$, we deduce that
$$\re(b_jz^d+P_j(z))\ge\frac{1}{4}|b_j|(|b_j|+|b_k|)^{(\nu/d)-1}|z|^\nu>2|z|^\alpha,$$
if $|z|$ is sufficiently large. This completes the proof.
\end{proof}

\section{Injectivity}   \label{sec_injectivity}

The aim of this section is to prove that $f$ is injective in certain disks contained in $\dC\setminus E_1$. We start with a basic injectivity criterion (see, e.g., \cite[Proposition 1.10]{pommerenke2}). 

\begin{lemma}  \label{lemma_injectivity_criterion_basic}
Let $D\subset\dC$ be a convex domain, and let $h:D\to\dC$ be holomorphic. If $\re(h'(z))>0$ for all $z\in D$, then $h$ is injective in $D$.
\end{lemma}

We also require the following criterion.

\begin{lemma}      \label{lemma_injectivity_criterion}
Let $z_0\in\dC$ and $r>0$. Let $h$ be holomorphic in $D(z_0,r)$. Suppose that $h'(\zeta)\ne0$ for all $\zeta\in D(z_0,r)$ and
$$\sup_{|\zeta-z_0|<r}\left|\dfrac{h''(\zeta)}{h'(\zeta)}\right|<\frac{1}{r}.$$
Then $h$ is injective in $D(z_0,r)$.
\end{lemma}

This follows directly from Becker's univalence criterion (see, e.g., \cite[Theorem 6.7]{pommerenke}). However, Lemma \ref{lemma_injectivity_criterion} may also be proved by much more elementary arguments using Lemma \ref{lemma_injectivity_criterion_basic}. We sketch the proof here.

\begin{proof}[Sketch of proof]
We may assume without loss of generality that $h'(z_0)=1$. Let $\psi$ be the branch of $\log h'$ in $D(z_0,r)$ satisfying $\psi(z_0)=0$. Then, for all $z\in D(z_0,r)$,
$$|\psi(z)|=\left|\int_{z_0}^z\frac{h''(\zeta)}{h'(\zeta)}\,d\zeta\right|<\frac{1}{r}|z-z_0|<1.$$
Thus, $\arg(h'(z))=\im(\psi(z))\in(-1,1)$. In particular, $\re(h'(z))>0$. Hence, $h$ is injective in $D(z_0,r)$ by Lemma \ref{lemma_injectivity_criterion_basic}.
\end{proof}

 We now state the main result of this section.

\begin{lemma}   \label{lemma_injectivity}
Let  $\sigma\in\left(0,\dfrac{1}{4d\max_j|b_j|}\right)$. Let $z\in\dC\setminus E_1$ such that $D(z,2\sigma|z|^{-(d-1)})\subset\dC\setminus E_1$. If $|z|$ is sufficiently large, then $f$ is injective in $D(z,2\sigma|z|^{-(d-1)})$. 
\end{lemma}

\begin{proof}
Let $r:=2\sigma|z|^{-(d-1)}$. By Lemma \ref{lemma_approx}, there exists $m\in\{1,...,N\}$ such that
$$\sup_{|\zeta-z|<r}\left|\frac{f''(\zeta)}{f'(\zeta)}\right|\le\frac{3}{2}d|b_m|\sup_{|\zeta-z|<r}|\zeta|^{d-1}\le2d|b_m||z|^{d-1}<\frac{1}{r},$$
if $|z|$ is sufficiently large. Thus, $f$ is injective in $D(z,r)$ by Lemma \ref{lemma_injectivity_criterion}.
\end{proof}

For $z\in\dC$ and $A\subset\dC$ let $\dist(z,A):=\inf\{|z-a|:\,a\in A\}$ denote the Euclidean distance of $z$ and $A$.

\begin{lemma}  \label{lemma_dist}
There is a constant $C_1>0$ such that, if $z\in\dC\setminus E_2$ and $|z|$ is sufficiently large, then 
$$\dist(z,E_1)\ge C_1|z|^{-3/2}.$$
\end{lemma}

The following Corollary is an immediate consequence of Lemma \ref{lemma_dist} and Lemma \ref{lemma_injectivity}.

\begin{cor}   \label{cor_injectivity}
If $z\in\dC\setminus E_2$ and $|z|$ is large, then $f$ is injective in $D(z, 2\sigma|z|^{-(d-1)})$,
for $\sigma$ as in Lemma \ref{lemma_injectivity}.
\end{cor}

\begin{proof}[Proof of Lemma \ref{lemma_dist}]
Let $z\in\dC\setminus E_2$. It suffices to show that if $|z|$ is sufficiently large, then 
$$\dist(z, E_1\cap D(z,1))\ge C_1|z|^{-3/2}$$
 for some constant $C_1>0$. Let $w\in E_1\cap D(z,1)$. Then there are $j,k\in\{1,...,N\}$ with $j\ne k$ such that $|\re(P_{j,k}(w))|\le|P_{j,k}(w)|^{\nu/d}$. If $|z|$ is sufficiently large, then $|w|\le\left(\frac{6}{5}\right)^{1/\nu}|z|$ and
\begin{align*}
|P_{j,k}(z)-P_{j,k}(w)|&\ge\re(P_{j,k}(z))-\re(P_{j,k}(w))\\
&\ge2|P_{j,k}(z)|^{\nu/d}-|P_{j,k}(w)|^{\nu/d}\\
&\ge2\cdot\frac{4}{5}|b_j-b_k|^{\nu/d}|z|^\nu-\frac{6}{5}|b_j-b_k|^{\nu/d}|w|^\nu\\
&\ge\frac{8}{5}|b_j-b_k|^{\nu/d}|z|^\nu-\left(\frac{6}{5}\right)^2|b_j-b_k|^{\nu/d}|z|^\nu\\
&=\frac{4}{25}|b_j-b_k|^{\nu/d}|z|^{d-5/2}.
\end{align*}
On the other hand,
\begin{align*}
|P_{j,k}(z)-P_{j,k}(w)|&=\left|\int_{w}^z P_{j,k}'(\zeta)\,d\zeta\right|\le\sup_{\zeta\in D(z,1)}|P_{j,k}'(\zeta)|\cdot|z-w|\\
&\le2d|b_j-b_k|(|z|+1)^{d-1}|z-w|\,\le4d|b_j-b_k||z|^{d-1}|z-w|.
\end{align*}
Thus,
$$|z-w|\ge\frac{|b_j-b_k|^{\nu/d-1}}{25d}|z|^{-3/2}\ge C_1|z|^{-3/2}$$
for $C_1=\dfrac{\min_{l\ne n}|b_l-b_n|^{\nu/d-1}}{25d}$.
\end{proof}

\section{Proof of Theorem \ref{thm_meas_detailed}}  \label{sec_proof}

In this section, we prove Theorem \ref{thm_meas_detailed}. First, we collect several results that we require. 
For $\alpha>0$ consider the function 
$$E_\alpha:[0,\infty)\to[0,\infty),\, E_\alpha(x)=\exp(x^\alpha).$$
We will use the following result \cite[Lemma 2.1]{bergweiler}.

\begin{lemma}   \label{lemma_iterate_exp}
Let $\beta>\alpha>0$. Then there exists $x_0>0$ such that 
$$E_\alpha^k(x)\ge E_\beta^{k-2}(x)\quad\text{for all } k\ge4\text{ and }x\ge x_0.$$
\end{lemma}

The next Lemma is due to Sixsmith \cite[Theorem 3.1]{sixsmith}.

\begin{lemma}   \label{lemma_Julia-criterion}
Let $h$ be a transcendental entire function and $z_0\in I(h)$. Let $z_n:=h^n(z_0)$ for all $n\in\dN$. Suppose that there exist $\lambda>1$ and $N\ge0$ such that 
$$h(z_n)\ne0\quad\text{and}\quad\left|z_n\frac{h'(z_n)}{h(z_n)}\right|\ge\lambda\quad\text{for all }n\ge N.$$
Then either $z_0$ is in a multiply connected Fatou component of $h$, or $z_0\in J(h)$.
\end{lemma}

The following result is due to Zheng \cite[Corollary 6 and Remark (J)]{zheng}.

\begin{lemma}   \label{lemma_mult-conn}
Let $h$ be a transcendental entire function of the form 
$$h(z)=\sum_{j=1}^Mq_j(z)\exp(p_j(z)),$$
with polynomials $p_j$ and $q_j$.  Then the Fatou set of $h$ has no multiply connected components.
\end{lemma}

For a curve $\gamma\subset\dC$ we denote by $\ell(\gamma)$ the Euclidean length of $\gamma$.

\begin{lemma}   \label{lemma_meas_curve}
Let $\gamma\subset\dC$ be a curve of positive length, and let $s\in(0,\ell(\gamma))$. Then
$$\meas(\{z:\,\dist(z,\gamma)\le s\})\le\frac{9\pi}{2}s\cdot \ell(\gamma).$$
\end{lemma}

\begin{proof}
Let $L\in\dN$ such that $L-1<\ell(\gamma)/s\le L$. We divide $\gamma$ into $L$ subcurves $\gamma_1,...,\gamma_L$ satisfying $\ell(\gamma_j)\le s$ for all $j\in\{1,...,L\}$. Then, for $j\in\{1,...,L\}$, there is an $a_j\in\dC$ such that $\gamma_j\subset\overline{D(a_j,s/2)}$. We have
$$\{z:\,\dist(z,\gamma_j)\le s\}\subset\overline{D\left(a_j,(3/2)s\right)}.$$
Thus,
$$\meas(\{z:\,\dist(z,\gamma_j)\le s\})\le\frac{9\pi}{4}s^2.$$
Using $(\ell(\gamma)/s)+1\le2\ell(\gamma)/s$, we get
$$\meas(\{z:\,\dist(z,\gamma)\le s\})\le L\frac{9\pi}{4}s^2<\left(\frac{\ell(\gamma)}{s}+1\right)\frac{9\pi}{4}s^2\le\frac{9\pi}{2}s\cdot\ell(\gamma).$$
\end{proof}

The next result is a direct consequence of the well-known Koebe distortion theorem (see, e.g., \cite[Theorem 1.6]{pommerenke}).

\begin{lemma}     \label{lemma_Koebe-distortion}
Let $z_0\in\dC$, $r>0$, and $\rho\in(0,1)$. Suppose that $h:D(z_0,r)\to\dC$ is holomorphic and injective. Then, for all $z\in D(z_0,\rho r)$,
$$\frac{|h'(z_0)|}{(1+\rho)^2}\le\frac{|h(z)-h(z_0)|}{|z-z_0|}\le\frac{|h'(z_0)|}{(1-\rho)^2}.$$
Moreover,
$$\frac{\max_{|z|\le\rho r}|h'(z)|}{\min_{|z|\le\rho r}|h'(z)|}\le\left(\frac{1+\rho}{1-\rho}\right)^4.$$
\end{lemma}

\begin{proof}[Proof of Theorem \ref{thm_meas_detailed}]
Let $B_0$ be a large open square centred at zero with sides parallel to the real and imaginary axis. Let $\tilde{\cS}$ be a collection of closed squares in $\dC$ with sides parallel to the real and imaginary axis such that
\begin{itemize}
\item $\bigcup_{S\in\tilde{\cS}}S=\dC\setminus B_0$,
\item for all $S_1, S_2\in\tilde{\cS}$ with $S_1\ne S_2$, we have $S_1^\circ\cap S_2^\circ=\emptyset$, and
\item for all $S\in\tilde{\cS}$, the side length $s$ of $S$ satisfies 
$$\frac{\sigma}{4\sqrt{2}\min_{z\in S}|z|^{d-1}}\le s\le \frac{\sigma}{\sqrt{2}\max_{z\in S}|z|^{d-1}},$$
with $\sigma$ as in Lemma \ref{lemma_injectivity}.
\end{itemize}

If the side length of $B_0$ is sufficiently large, this can be achieved as follows. First, divide $\dC\setminus B_0$ into squares of a fixed size so that the side length of all squares satisfies the lower bound. If the side length $s$ of a square does not satisfy the upper bound, divide it into four squares of side length $s/2$, and then continue this procedure until the side length of the squares satisfies the upper bound.

Let $\cS$ be the collection of all $S\in \tilde{\cS}$ such that $\dist(S,E_1)>\dfrac{2\sigma}{\min_{z\in S}|z|^{d-1}}$. By Lemma \ref{lemma_dist} and the definition of $\cS$, 
\begin{equation}   \label{eq_location_S}
\dC\setminus (E_2\cup B_0)\subset\bigcup_{S\in\cS}S\subset\dC\setminus (E_1\cup B_0),
\end{equation}
if $B_0$ is sufficiently large.

Next, we construct a subset of $A(f)\cap J(f)$ as an intersection of nested sets. Fix a square $S_0\in\cS$. Let 
$$\cK_0:=\{S_0\}$$
and, for $n\in\dN$, let 
$$\cK_n:=\{T_n\subset S_0:\, f^n(T_n)\in\cS \text{ and } T_n\subset T_{n-1} \text{ for some } T_{n-1}\in\cK_{n-1}\}.$$
We first show that
$$T:=\bigcap_{n\in\dN_0}\left(\bigcup_{T_n\in\cK_n}T_n\right)\subset A(f)\cap J(f).$$
To do so, let $z\in T$. Then $f^n(z)\in\dC\setminus E_1$ for all $n\in\dN_0$. Let $\alpha\in(0,\nu)$ and $\beta>d$. Using Lemma \ref{lemma_flarge}, Lemma \ref{lemma_iterate_exp}, and the fact that $M(r,f)\le E_\beta(r)$ for all large $r$, we get 
$$|f^n(z)|\ge E_\alpha^n(|z|)\ge E_\beta^{n-2}(|z|)\ge M^{n-2}(|z|,f)$$
for all $n\ge4$, if $|z|$ is sufficiently large. This yields $z\in A(f)$. For $n\in\dN$, let $z_n:=f^n(z)$. By Lemma \ref{lemma_approx} and Lemma \ref{lemma_flarge},
$$\left|z_n\frac{f'(z_n)}{f(z_n)}\right|\ge\frac{1}{2}d\min_k|b_k||z_n|^d>2,$$
if $|z|$ is large. By Lemma \ref{lemma_Julia-criterion} and Lemma \ref{lemma_mult-conn}, $z\in J(f)$. Thus, $T\subset A(f)\cap J(f)$, provided the square $B_0$ is chosen sufficiently large. 

For $A\subset\dC$ define
$$\pack(A):=\{S\in\cS:\,S\subset A\}.$$
Moreover, for $A,B\subset\dC$ with $\meas(B)>0$, let
$$\dens(A,B):=\frac{\meas(A\cap B)}{\meas(B)}$$
denote the \textit{density} of $A$ in $B$. We will show that for any square $S\in\cS$,
$$\dens\left(f(S)\setminus\bigcup_{S'\in\pack(f(S))}S',\,f(S)\right)\le\exp\left(-\frac{1}{2}\min_{z\in S}|z|^\alpha\right),$$
where, as before, $0<\alpha<\nu$. See Figure \ref{fig_pack} for an illustration of $\pack(f(S))$.

\begin{figure}[ht]     
\centering 
\includegraphics[width=0.3\textwidth]{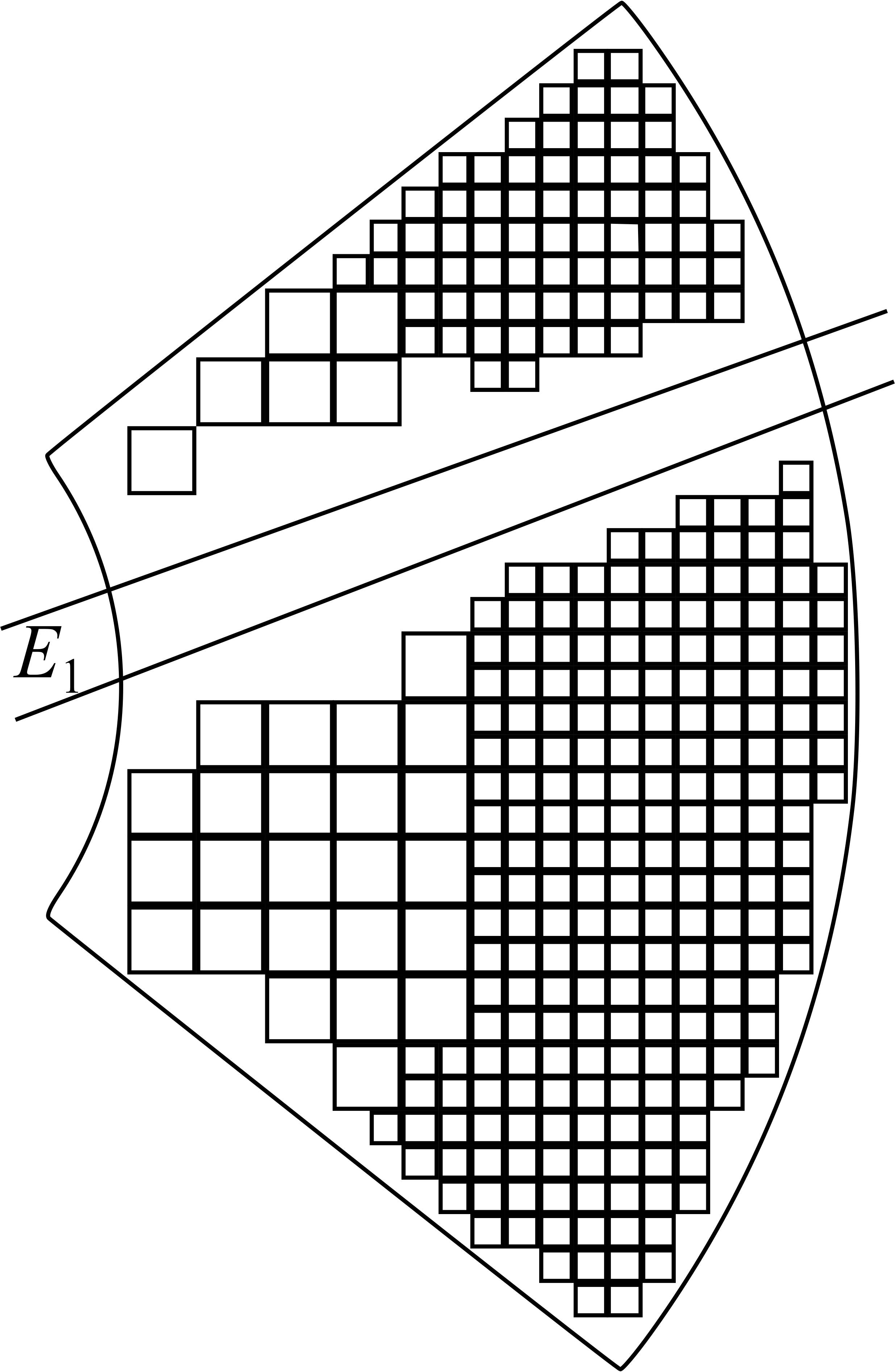}

\caption{An illustration of $\pack(f(S))$ (not to scale).}
\label{fig_pack}
\end{figure}

By Lemma \ref{lemma_injectivity}, $f$ is injective in $S$. We have
$$\meas(f(S))=\int_S|f'(z)|^2d(x,y)\ge\meas(S)\cdot\min_{z\in S}|f'(z)|^2.$$
Moreover, 
$$f(S)\subset\bigcup_{S'\in\tilde\cS}S'.$$ 
By \eqref{eq_location_S} and Lemma \ref{lemma_measE2}, the Lebesgue measure of the union of all squares $S'\in\tilde{\cS}\setminus\cS$ is at most $\meas(E_2)<\infty$. We now consider the union of all squares $S'\in\cS$ with $S'\cap\partial f(S)\ne\emptyset$. The length of $\partial f(S)$ satisfies
$$\ell(\partial f(S))=\int_{\partial S}|f'(z)|\,|dz|\le\max_{z\in S}|f'(z)|\ell(\partial S)\le\frac{4\sigma}{\sqrt{2}\max_{z\in S}|z|^{d-1}}\max_{z\in S}|f'(z)|.$$
Analogously,
$$\ell(\partial f(S))\ge\frac{\sigma}{\sqrt{2}\min_{z\in S}|z|^{d-1}}\min_{z\in S}|f'(z)|.$$
By Lemma \ref{lemma_flarge}, $|f'(z)|\ge\exp(|z|^\alpha)$ for all $z\in S$. In particular, $\ell(\partial f(S))>1$.
For $S'\in\cS$ with $S'\cap\partial f(S)\ne\emptyset$, we have 
$$S'\subset\{z:\,\dist(z,\partial f(S))\le1\}.$$
By Lemma \ref{lemma_meas_curve},
$$\meas(\{z:\,\dist(z,\partial f(S))\le1\})\le\frac{9\pi}{2}\cdot\ell(\gamma)\le\frac{18\pi\sigma}{\sqrt{2}\max_{z\in S}|z|^{d-1}}\max_{z\in S}|f'(z)|.$$
Altogether, we get
\begin{align*}
\meas\left(f(S)\setminus\bigcup_{S'\in\pack(f(S))}S'\right)&\le\meas(E_2)+\dfrac{18\pi\sigma}{\sqrt{2}\max_{z\in S}|z|^{d-1}}\max_{z\in S}|f'(z)|\\
&\le\max_{z\in S}|f'(z)|,
\end{align*}
provided the square $B_0$ is sufficiently large.
Thus,
$$\dens\left(f(S)\setminus\bigcup_{S'\in\pack(f(S))}S',\,f(S)\right)\le\frac{\max_{z\in S}|f'(z)|}{\meas(S)\cdot\min_{z\in S}|f'(z)|^2}.$$
Let $z_0$ be the centre of $S$. Then $S\subset D\left(z_0,(\sigma/2)|z_0|^{-(d-1)}\right)$, and by Lemma \ref{lemma_injectivity}, $f$ is injective in $D(z_0,2\sigma|z_0|^{-(d-1)})$.
By the Koebe distortion theorem (Lemma \ref{lemma_Koebe-distortion}) and Lemma \ref{lemma_flarge},
\begin{align}   \label{eq_dens_f}
\dens\left(f(S)\setminus\bigcup_{S'\in\pack(f(S))}S',\,f(S)\right)&\le\left(\frac{5}{3}\right)^4\frac{1}{\meas(S)\min_{z\in S}|f'(z)|}\nonumber\\
&\le\left(\frac{5}{3}\right)^4\frac{32\min_{z\in S}|z|^{2(d-1)}}{\sigma^2\exp\left(\min_{z\in S}|z|^\alpha\right)}\nonumber\\
&\le\exp\left(-\frac{1}{2}\min_{z\in S}|z|^\alpha\right),
\end{align}
if $B_0$ is large.

Let $n\in\dN_0$ and $T_n\in\cK_n$. Then $f^n(T_n)\in\cS$. By \eqref{eq_dens_f} applied to $S=f^n(T_n)$ and Lemma \ref{lemma_flarge},
\begin{align*}  
&\dens\left(f^{n+1}(T_n)\setminus\bigcup_{S'\in\pack(f^{n+1}(T_n))}S',\,f^{n+1}(T_n)\right)\\
&\qquad\le\exp\left(-\frac{1}{2}\min_{z\in f^n(T_n)}|z|^\alpha\right)
\le\displaystyle\exp\left(-\frac{1}{2}\Big(E_\alpha^n\Big(\min_{z\in S_0}|z|\Big)\Big)^\alpha\right).
\end{align*}

We use this to prove that the set $T_n\setminus\left(\bigcup_{T_{n+1}\in\cK_{n+1}}T_{n+1}\right)$ has small density in $T_n$.
For all $k\in\{1,...,n\}$, there is a square $S_k\in\cS$ such that $f^k(T_n)\subset S_k$. In particular, $f^{n+1}$ is injective in $T_n$. Thus,

\begin{align*}
&\meas\left(T_n\setminus \bigcup_{T_{n+1}\in\cK_{n+1}}T_{n+1}\right)\\
&\qquad\le\frac{1}{\displaystyle\min_{z\in T_n}|(f^{n+1})'(z)|^2}\cdot\meas\left(f^{n+1}(T_n)\setminus\bigcup_{S'\in\pack(f^{n+1}(T_n))}S'\right),
\end{align*}
and
$$\meas(T_n)\ge\frac{1}{\displaystyle\max_{z\in T_n}|(f^{n+1})'(z)|^2}\meas(f^{n+1}(T_n)).$$
Hence,
\begin{align}   \label{eq_dens_Tn}
&\dens\left(T_n\setminus \bigcup_{T_{n+1}\in\cK_{n+1}}T_{n+1},\,T_n\right)\nonumber\\
&\qquad\le\frac{\max_{z\in T_n}|(f^{n+1})'(z)|^2}{\min_{z\in T_n}|(f^{n+1})'(z)|^2}\dens\left(f^{n+1}(T_n)\setminus\bigcup_{S'\in\pack(f^{n+1}(T_n))}S',\,f^{n+1}(T_n)\right)\nonumber\\
&\qquad\le\left(\frac{\max_{z\in T_n}|(f^{n+1})'(z)|}{\min_{z\in T_n}|(f^{n+1})'(z)|}\right)^2\exp\left(-\frac{1}{2}\Big(E_\alpha^n\Big(\min_{z\in S_0}|z|\Big)\Big)^\alpha\right).
\end{align}
To estimate $\dfrac{\max_{z\in T_n}|(f^{n+1})'(z)|}{\min_{z\in T_n}|(f^{n+1})'(z)|}$, let $w_0\in f^k(T_n)$. Then
$$f^k(T_n)\subset S_k\subset\overline{D\left(w_0, \sigma|w_0|^{-(d-1)}\right)}.$$
By Lemma \ref{lemma_injectivity}, $f$ is injective in $D\left(w_0, 2\sigma|w_0|^{-(d-1)}\right)$. The Koebe distortion theorem (Lemma \ref{lemma_Koebe-distortion}) yields that, for all $w\in f^k(T_n)$, 
$$\frac{|f(w)-f(w_0)|}{|w-w_0|}\ge\frac{4}{9}|f'(w_0)|.$$
By Lemma \ref{lemma_flarge}, $|f'(w_0)|\ge5$. Thus,
$$\diam(f^{k+1}(T_n))>2\diam(f^k(T_n)).$$
Induction yields
\begin{align*}
\diam(f^k(T_n))&<\frac{1}{2^{n-k}}\diam(f^n(T_n))\le\frac{1}{2^{n-k}}\cdot\frac{\sigma}{\max_{z\in f^n(S_0)}|z|^{d-1}}\\
&\le\frac{1}{2^{n-k}}\cdot\frac{\sigma}{\max_{z\in f^k(S_0)}|z|^{d-1}}.
\end{align*}

In particular, for $z_k\in f^k(T_n)$, 
$$f^k(T_n)\subset\overline{D\left(z_k,\frac{2^{k-n}\sigma}{\max_{z\in f^k(S_0)}|z|^{d-1}}\right)}\subset\overline{D\left(z_k,2^{k-n}\sigma|z_k|^{-(d-1)}\right)}.$$
Since $f$ is injective in $D\left(z_k,2\sigma|z_k|^{-(d-1)}\right)$, the Koebe distortion theorem (Lemma \ref{lemma_Koebe-distortion}) yields
$$\frac{\max_{z\in f^k(T_n)}|f'(z)|}{\min_{z\in f^k(T_n)}|f'(z)|}\le\left(\frac{1+2^{k-n-1}}{1-2^{k-n-1}}\right)^4.$$
Since $(f^{n+1})'(z)=\prod_{k=0}^nf'(f^k(z))$, we get
\begin{align*}
\frac{\max_{z\in T_n}|(f^{n+1})'(z)|}{\min_{z\in T_n}|(f^{n+1})'(z)|}&\le\prod_{k=0}^n\Bigg(\frac{1+2^{k-n-1}}{1-2^{k-n-1}}\Bigg)^4=\Bigg(\prod_{j=1}^{n+1}\frac{1+2^{-j}}{1-2^{-j}}\Bigg)^4\\
&\le\Bigg(\prod_{j=1}^\infty \frac{1+2^{-j}}{1-2^{-j}}\Bigg)^4=:C_2,
\end{align*}
where $C_2\in(0,\infty)$. Together with \eqref{eq_dens_Tn}, this implies
$$\dens\left(T_n\setminus\bigcup_{T_{n+1}\in\cK_{n+1}}T_{n+1},\,T_n\right)\le C_2^2\exp\left(-\frac{1}{2}\Big(E_\alpha^n\Big(\min_{z\in S_0}|z|\Big)\Big)^\alpha\right).$$
Thus,
\begin{align*}
\meas(S_0\setminus T)&=\sum_{n=0}^\infty\sum_{T_n\in\cK_n}\meas\left(T_n\setminus\bigcup_{T_{n+1}\in\cK_{n+1}}T_{n+1}\right)\\
&=\sum_{n=0}^\infty\sum_{T_n\in\cK_n}\dens\left(T_n\setminus\bigcup_{T_{n+1}\in\cK_{n+1}}T_{n+1},\,T_n\right)\cdot\meas(T_n)\\
&\le\sum_{n=0}^\infty C_2^2\exp\left(-\frac{1}{2}\Big(E_\alpha^n\Big(\min_{z\in S_0}|z|\Big)\Big)^\alpha\right)\sum_{T_n\in\cK_n}\meas(T_n)\\
&\le C_2^2\sum_{n=0}^\infty\exp\left(-\frac{1}{2}\Big(E_\alpha^n\Big(\min_{z\in S_0}|z|\Big)\Big)^\alpha\right)\cdot\meas(S_0).
\end{align*}
For all large $x$, $\exp(\alpha x^\alpha)\ge x^\alpha+2\log2$, and thus
$$\exp\left(-\frac{1}{2}(E_\alpha(x))^\alpha\right)=\exp\left(-\frac{1}{2}\exp(\alpha x^\alpha)\right)\le\frac{1}{2}\exp\left(-\frac{1}{2}x^\alpha\right).$$
Induction yields
$$\sum_{n=0}^\infty \exp\left(-\frac{1}{2}\Big(E_\alpha^n\Big(\min_{z\in S_0}|z|\Big)\Big)^\alpha\right)\le\sum_{n=0}^\infty\frac{1}{2^n}\exp\left(-\frac{1}{2}\min_{z\in S_0}|z|^\alpha\right)=2\exp\left(-\frac{1}{2}\min_{z\in S_0}|z|^\alpha\right).$$
Thus,
$$\meas(S_0\setminus(A(f)\cap J(f)))\le\meas(S_0\setminus T)\le2C_2^2\exp\left(-\frac{1}{2}\min_{z\in S_0}|z|^\alpha\right)\meas(S_0).$$
To conclude that $\meas(\dC\setminus(A(f)\cap J(f)))<\infty$, fix $R>0$ such that $B_0\subset D(0,R/2)$. For $r\ge R$ define
$$\ann(r):=\{z:\,r\le|z|\le2r\}.$$
Then
$$\bigcup_{\substack{S_0\in\cS\\S_0\cap\ann(r)\ne\emptyset}}S_0\subset\left\{z:\,\frac{r}{2}\le|z|\le3r\right\}\subset\overline{D(0,3r)}.$$
We get
\begin{align*}
&\meas(\ann(r)\setminus((A(f)\cap J(f))\cup E_2))\le\sum_{\substack{S_0\in\cS\\S_0\cap\ann(r)\ne\emptyset}}2C_2^2\exp\left(-\frac{1}{2}\min_{z\in S_0}|z|^\alpha\right)\meas(S_0)\\
&\qquad\le2C_2^2\exp\left(-\frac{1}{2}\left(\frac{r}{2}\right)^\alpha\right)\sum_{\substack{S_0\in\cS\\S_0\cap\ann(r)\ne\emptyset}}\meas(S_0)\le2C_2^2\exp\left(-\frac{r^\alpha}{2^{1+\alpha}}\right)\meas(D(0,3r))\\
&\qquad =18\pi C_2^2 r^2\exp\left(-\frac{r^\alpha}{2^{1+\alpha}}\right)\le\exp\left(-\frac{r^\alpha}{2^{2+\alpha}}\right),
\end{align*}
if $r$ is sufficiently large. Applying this to the annuli $\ann(2^nR)$, $n\in\dN_0$, yields
\begin{align*}
&\meas(\dC\setminus(A(f)\cap J(f)))\\
&\quad\le\meas(D(0,R))+\meas(E_2)+\sum_{n=0}^\infty\meas(\ann(2^nR)\setminus((A(f)\cap J(f))\cup E_2))\\
&\quad\le\meas(D(0,R))+\meas(E_2)+\sum_{n=0}^\infty\exp\left(-\frac{R^\alpha}{2^{2+\alpha}}2^{n\alpha}\right)<\infty.
\end{align*}
This completes the proof.
\end{proof}

\section{Verification of the properties of Example \ref{ex_counterex}}   \label{sec_counterexample}

In this section, we consider the function
$$h(z)=\frac{1}{2}\exp(z^3+iz)-\frac{1}{2}\exp(-z^3+iz)=\exp(iz)\sinh(z^3)$$
given in Example \ref{ex_counterex}. Note that $h$ satisfies all assumptions of Theorem \ref{thm_meas_detailed} except for condition \eqref{eq_extra-condition}. Moreover, the function $h$ has a superattracting fixed point at zero. Recall that we want to prove that the attractive basin of zero has infinite Lebesgue measure, so that in particular the Lebesgue measure of $\dC\setminus(J(h)\cap A(h))$ is infinite.\\

To do so, fix $\varepsilon>0$ such that $D(0,\varepsilon)$ is contained in the attractive basin of zero. For large $r_0>1$, let 
$$B:=\left\{re^{i\theta}:\,r\ge r_0,\,\left|\theta-\frac{\pi}{2}\right|\le\frac{1}{r^2\log r}\right\}.$$
We have
$$\meas(B)=\int_{r_0}^\infty\int_{\frac{\pi}{2}-\frac{1}{r^2\log r}}^{\frac{\pi}{2}+\frac{1}{r^2\log r}}r\,d\theta d r=\int_{r_0}^\infty \frac{2}{r\log r} dr=2\int_{\log(r_0)}^\infty \frac{1}{u}\, du=\infty.$$
We now show that if $r_0$ is sufficiently large, then $h(B)\subset D(0,\varepsilon)$, and hence $B$ is contained in the attractive basin of zero.
Let $z=re^{i\theta}\in B$. Then
\begin{align*}
|h(z)|&\le|\exp(iz)|\cdot\frac{1}{2}\left(\left|\exp\left(z^3\right)\right|+\left|\exp\left(-z^3\right)\right|\right)\\
&=\exp\left(r\cos\left(\theta+\frac{\pi}{2}\right)\right)\cdot\frac{1}{2}\left(\exp\left(r^3\cos(3\theta)\right)+\exp\left(-r^3\cos(3\theta)\right)\right)\\
&\le\exp\left(r\cos\left(\theta+\frac{\pi}{2}\right)\right)\cdot\exp\left(r^3|\cos(3\theta)|\right).
\end{align*}
We have
$$\cos(w)=-1+\cO\left((w-\pi)^2\right)\qquad\text{as }w\to\pi,$$
and
$$\cos(w)=\left(w-\frac{3\pi}{2}\right)+\cO\left(\left(w-\frac{3\pi}{2}\right)^3\right)\qquad\text{as }w\to\frac{3\pi}{2}.$$
Since $\left|\theta-\dfrac{\pi}{2}\right|\le\dfrac{1}{r^2\log r}$, this implies
$$\cos\left(\theta+\frac{\pi}{2}\right)\le-\frac{1}{2}$$
and
$$|\cos(3\theta)|\le2\cdot\frac{3}{r^2\log r},$$
if $r$ is sufficiently large. 
Thus,
$$|h(z)|\le\exp\left(-\frac{1}{2}r\right)\cdot\exp\left(6\cdot \frac{r}{\log r}\right)\le\exp\left(-\frac{1}{4}r\right)<\varepsilon,$$
if $r$ is sufficiently large. So $h(B)\subset D(0,\varepsilon)\subset F(h)\setminus I(h)$, if $r_0$ is sufficiently large.

\paragraph{Acknowledgements} 
I would like to thank Walter Bergweiler for helpful suggestions.

\end{document}